\documentclass[reqno,12pt]{amsart}
\usepackage{amssymb,amsmath,amsfonts,amsthm}
\begin{document}


\newcommand\bbb{\ensuremath{\mathbb{B}}}
\newcommand\bbc{\ensuremath{\mathbb{C}}}
\newcommand\bbd{\ensuremath{\mathbb{D}}}
\newcommand\bbf{\ensuremath{\mathbb{F}}}
\newcommand\bbk{\ensuremath{\mathbb{K}}}
\newcommand\bbn{\ensuremath{\mathbb{N}}}
\newcommand\bbp{\ensuremath{\mathbb{P}}}
\newcommand\bbq{\ensuremath{\mathbb{Q}}}
\newcommand\bbr{\ensuremath{\mathbb{R}}}
\newcommand\bbt{\ensuremath{\mathbb{T}}}
\newcommand\bbz{\ensuremath{\mathbb{Z}}}

\newcommand\bbbs{\ensuremath{\mathbb{B}\text{ }}}
\newcommand\bbcs{\ensuremath{\mathbb{C}\text{ }}}
\newcommand\bbds{\ensuremath{\mathbb{D}\text{ }}}
\newcommand\bbfs{\ensuremath{\mathbb{F}\text{ }}}
\newcommand\bbks{\ensuremath{\mathbb{K}\text{ }}}
\newcommand\bbns{\ensuremath{\mathbb{N}\text{ }}}
\newcommand\bbps{\ensuremath{\mathbb{P}\text{ }}}
\newcommand\bbqs{\ensuremath{\mathbb{Q}\text{ }}}
\newcommand\bbrs{\ensuremath{\mathbb{R}\text{ }}}
\newcommand\bbts{\ensuremath{\mathbb{T}\text{ }}}
\newcommand\bbzs{\ensuremath{\mathbb{Z}\text{ }}}

\newcommand\mca{\ensuremath{\mathcal{A}}}
\newcommand\mcb{\ensuremath{\mathcal{B}}}
\newcommand\mcc{\ensuremath{\mathcal{C}}}
\newcommand\mcd{\ensuremath{\mathcal{D}}}
\newcommand\mce{\ensuremath{\mathcal{E}}}
\newcommand\mcf{\ensuremath{\mathcal{F}}}
\newcommand\mch{\ensuremath{\mathcal{H}}}
\newcommand\mck{\ensuremath{\mathcal{K}}}
\newcommand\mcl{\ensuremath{\mathcal{L}}}
\newcommand\mcm{\ensuremath{\mathcal{M}}}
\newcommand\mcn{\ensuremath{\mathcal{N}}}
\newcommand\mcp{\ensuremath{\mathcal{P}}}
\newcommand\mcs{\ensuremath{\mathcal{S}}}
\newcommand\mct{\ensuremath{\mathcal{T}}}
\newcommand\mcu{\ensuremath{\mathcal{U}}}
\newcommand\mcw{\ensuremath{\mathcal{W}}}
\newcommand\mcx{\ensuremath{\mathcal{X}}}
\newcommand\mcy{\ensuremath{\mathcal{Y}}}
\newcommand\mcas{\ensuremath{\mathcal{A}\text{ }}}
\newcommand\mcbs{\ensuremath{\mathcal{B}\text{ }}}
\newcommand\mccs{\ensuremath{\mathcal{C}\text{ }}}
\newcommand\mcds{\ensuremath{\mathcal{D}\text{ }}}
\newcommand\mces{\ensuremath{\mathcal{E}\text{ }}}
\newcommand\mcfs{\ensuremath{\mathcal{F}\text{ }}}
\newcommand\mchs{\ensuremath{\mathcal{H}\text{ }}}
\newcommand\mcks{\ensuremath{\mathcal{K}\text{ }}}
\newcommand\mcls{\ensuremath{\mathcal{L}\text{ }}}
\newcommand\mcms{\ensuremath{\mathcal{M}\text{ }}}
\newcommand\mcns{\ensuremath{\mathcal{N}\text{ }}}
\newcommand\mcps{\ensuremath{\mathcal{P}\text{ }}}
\newcommand\mcss{\ensuremath{\mathcal{S}\text{ }}}
\newcommand\mcts{\ensuremath{\mathcal{T}\text{ }}}
\newcommand\mcus{\ensuremath{\mathcal{U}\text{ }}}
\newcommand\mcws{\ensuremath{\mathcal{W}\text{ }}}
\newcommand\mcxs{\ensuremath{\mathcal{X}\text{ }}}
\newcommand\mcys{\ensuremath{\mathcal{Y}\text{ }}}

\newcommand\mfa{\ensuremath{\mathfrak{A}}}
\newcommand\mfb{\ensuremath{\mathfrak{B}}}
\newcommand\mfc{\ensuremath{\mathfrak{C}}}
\newcommand\mfg{\ensuremath{\mathfrak{g}}}
\newcommand\mfi{\ensuremath{\mathfrak{I}}}
\newcommand\mfm{\ensuremath{\mathfrak{M}}}
\newcommand\mfs{\ensuremath{\mathfrak{S}}}
\newcommand\mfx{\ensuremath{\mathfrak{X}}}
\newcommand\mfy{\ensuremath{\mathfrak{Y}}}
\newcommand\mfz{\ensuremath{\mathfrak{Z}}}
\newcommand\mfas{\ensuremath{\mathfrak{A}\text{ }}}
\newcommand\mfbs{\ensuremath{\mathfrak{B}\text{ }}}
\newcommand\mfcs{\ensuremath{\mathfrak{C}\text{ }}}
\newcommand\mfgs{\ensuremath{\mathfrak{g}\text{ }}}
\newcommand\mfis{\ensuremath{\mathfrak{I}\text{ }}}
\newcommand\mfms{\ensuremath{\mathfrak{M}\text{ }}}
\newcommand\mfss{\ensuremath{\mathfrak{S}\text{ }}}
\newcommand\mfxs{\ensuremath{\mathfrak{X}\text{ }}}
\newcommand\mfys{\ensuremath{\mathfrak{Y}\text{ }}}
\newcommand\mfzs{\ensuremath{\mathfrak{Z}\text{ }}}

\newcommand\bnd{\ensuremath{\mathcal{B(H)}}}
\newcommand\kpt{\ensuremath{\mathcal{K(H)}}}
\newcommand\fnt{\ensuremath{\mathcal{F(H)}}}
\newcommand\fnts{\ensuremath{\mathcal{F(H)}\text{ }}}
\newcommand\kpts{\ensuremath{\mathcal{K(H)}\text{ }}}
\newcommand\bnds{\ensuremath{\mathcal{B(H)}\text{ }}}

\newcommand\bndx{\ensuremath{\mathcal{B}\mathfrak{(X)}}}
\newcommand\kptx{\ensuremath{\mathcal{K}\mathfrak{(X)}}}
\newcommand\fntx{\ensuremath{\mathcal{F}\mathfrak{(X)}}}
\newcommand\fntxs{\ensuremath{\mathcal{F}\mathfrak{(X)}\text{ }}}
\newcommand\kptxs{\ensuremath{\mathcal{K}\mathfrak{(X)}\text{ }}}
\newcommand\bndxs{\ensuremath{\mathcal{B}\mathfrak{(X)}\text{ }}}

\newcommand\bndxy{\ensuremath{\mathcal{B}\mathfrak{(X,Y)}}}
\newcommand\kptxy{\ensuremath{\mathcal{K}\mathfrak{(X,Y)}}}
\newcommand\fntxy{\ensuremath{\mathcal{F}\mathfrak{(X,Y)}}}
\newcommand\fntxys{\ensuremath{\mathcal{F}\mathfrak{(X,Y)}\text{ }}}
\newcommand\kptxys{\ensuremath{\mathcal{K}\mathfrak{(X,Y)}\text{ }}}
\newcommand\bndxys{\ensuremath{\mathcal{B}\mathfrak{(X,Y)}\text{ }}}

\newcommand\Ran{\text{Ran}}

\newcommand\spr{\text{spr}}
\newcommand\co{\text{co}}

\newcommand\floor[1]{\left\lfloor #1\right\rfloor}
\newcommand\ceil[1]{\left\lceil #1\right\rceil}

\newcommand\ddt[2]{\ensuremath{\frac{d #1}{d #2}}}
\newcommand\pdt[2]{\ensuremath{\frac{\partial #1}{\partial #2}}}

\newcommand\ip[2]{\ensuremath{\left\langle #1,#2\right\rangle}}

\newcommand\half{\ensuremath{\frac{1}{2}}}
\newcommand\spn{\ensuremath{\text{span}}}
\newcommand\fdim{\ensuremath{\text{fdim}}}
\newcommand{\tnorm}[1]{%
  \left\vert\kern-0.9pt\left\vert\kern-0.9pt\left\vert #1
    \right\vert\kern-0.9pt\right\vert\kern-0.9pt\right\vert}

\newtheorem{theorem}{Theorem}[section]
\newtheorem{lemma}[theorem]{Lemma}
\newtheorem{cor}[theorem]{Corollary}
\newtheorem{propn}[theorem]{Proposition}

\theoremstyle{definition}
\newtheorem{eg}[theorem]{Example}
\newtheorem{defn}[theorem]{Definition}

\title[free products of von Neumann algebras]{The amalgamated free product of hyperfinite von Neumann algebras over finite dimensional subalgebras}

\author[Dykema]{Kenneth J. Dykema}
\author[Redelmeier]{Daniel Redelmeier}

\address{Department of Mathematics, Texas A\&M University,
College Station, TX 77843-3368, USA}
\email{kdykema@math.tamu.edu, danielr@math.tamu.edu}
\thanks{Research supported in part by NSF grant DMS--0901220}
\maketitle

\bibliographystyle{amsplain}

\begin{abstract}  In this paper we describe the amalgamated free product of two hyperfinite von Neumann algebras over a finite dimensional subalgebra.  In general the free product is of the form $H\oplus\bigoplus_{i=1}^{N}F_{i}$ where $F_{i}$ are interpolated free group factors and $H$ is a hyperfinite von Neumann algebra.  We then show that the class of von Neumann algebras of this form is closed under taking amalgamated free products over finite dimensional subalgebras.
\end{abstract}
\section{Introduction}

The (reduced) amalgamated free products of C$^{*}$-algebras and of von Neumann algebras have proved to be useful constructions since first introduced by Voiculescu in \cite{voicAFP} and used by Popa in \cite{popaAFP}.  For $A$ and $B$ finite von Neumann algebras, and $D$ a unital subalgebra of both $A$ and $B$, with trace preserving expectations $E^{A}_{D}$ and $E^{B}_{D}$, we denote the amalgamated free product of $A$ and $B$ over $D$ by $M=A*_{D}B$.  In this paper we describe $M$ in the case where $A$ and $B$ are hyperfinite and $D$ is finite dimensional.  This extends the results in \cite{kenDuke}, which examines the free product over the scalars of hyperfinite von Neumann algebras, and in \cite{kenAmJM}, which examines the amalgamated free product of multimatrix algebras.

With Theorem \ref{closedtheorem} we also extend Theorem 4.4 in \cite{kenLMS},
which considered similar amalgamated free products but with additional restrictions on $A$ and $B$
and which was used by Kodiyalam and Sunder in \cite{kodsund} and in the related paper by Guionnet,  Jones, and Shlyakhtenko in \cite{GJS}.

\section{Basic Theorems and Definitions}
We consider all von Neumann algebras to be equipped with a specified normal faithful tracial state.  We will often write von Neumann algebras in the following way:
\[
\mcn=\overset{p_{1}}{A_{1}}\oplus \overset{p_{2}}{A_{2}}\oplus\dots,
\]
where $p_{i}$ denotes the central support of $A_{i}$ in $\mcn$.  If a summand is a matrix algebra, we use the notation $\underset{t}{M_{n}}$ to indicate that the trace of a minimal projection in $M_{n}$ is $t$.  As mentioned in the introduction, $A*_{D} B$ denotes the amalgamated free product of von Neumann algebras $A$ and $B$ over $D$ with respect to a trace preserving conditional expectation onto $D$.   When we refer to the expectation we mean the expectation onto $D$ unless otherwise specified.

Our results are in terms of the interpolated free group factors that were introduced
independently in \cite{kenIFGF} and by Radulescu in \cite{rad}.  We shall
denote these by $L(F_{s})$, where $1<s\leq\infty$.  If $s$ is an integer then it is
the corresponding free group factor.  As one would hope, these satisfy
$L(F_{s})*L(F_{s'})=L(F_{s+s'})$, and we think of $s$ as the (not necessarily integer) number
of free generators of $L(F_s)$.  When an interpolated free group factor is
compressed by a projection of trace $t$ or is dilated, we have
\[
\left(L(F_{s})\right)_{t}=L(F_{1+(s-1)/t^{2}}),
\]
$0<t<\infty$.  It is known that either $L(F_{s})\cong L(F_{s'})$ for all $s$ and $s'$
or $L(F_{s})\ncong L(F_{s'})$ for all $s\ne s'$, though it remains open which of the two holds.
This is the present state of knowledge about the isomorphism problem for free group factors.

If it turns out that the interpolated free group factors are not isomorphic to each other,
than then it will be interesting to keep track of the number of generators whenever free group factors appear
in constructions or results.
For this purpose the (perhaps confusingly named) notion of \emph{free dimension} was
introduced in \cite{kenDuke}.
Ostensibly, the free dimension is a number assigned to certain finite von Neumann algebras
endowed with normal, faithful, tracial states.
For a finite von Neumann algebra $M$ with given trace, we will denote the
free dimension with respect to that trace by $\fdim(M)$ (thus, the name of the trace is suppressed in the notation).
The free dimension is defined by the following conventions:
\begin{enumerate}
\item \label{fdim:it1}  For $M=L(F_{s})$, $\fdim(M)=s$.
\item  For $M=\oplus_{i\in I}\underset{t_{i}}{M_{n_{i}}}$, $\fdim(M)=1-\sum_{i\in I}t_{i}^{2}$.
\item  For $M$ a diffuse hyperfinite von Neumann algebra, $\fdim(M)=1$.
\item  If 
\[
M=F_{0}\oplus\bigoplus_{i\in I}\overset{p_{i}}{L(F_{s_{i}})}\oplus\bigoplus_{j\in J}\underset{t_{j}}{M_{n_{j}}},
\]
where $F_{0}$ is a diffuse hyperfinite von Neumann algebra, or zero, and $I$ and $J$ may be empty, then,
\[
\fdim(M)=1+\left(\sum_{i\in I}\tau(p_{i})^{2}(s_{i}-1)\right)-\sum_{j\in J}t_{j}^{2}.
\]
\end{enumerate}
In fact, as is apparent from item~(\ref{fdim:it1}), we do not know if $\fdim(M)$ is in all cases well defined.
There are several ways around this problem.
One is to say, for the purpose of interpreting results, we only care about the free dimension if the free group
factors are not isomorphic to each other, in which case it is well defined.
Another more cumbersome but mathematically correct approach is to speak only of free dimension for
generating sets of finite von Neumann algebras.
This approach was taken in~\cite{kenLMS} and earlier in~\cite{kenJFA02}, and corresponds well to results
about Voiculescu's free entropy dimension (which is intrinsically a different notion,
see~\cite{VoiIII}, \cite{VoiStr} and \cite{JungHyperf}).
See section~3 of~\cite{kenLMS}
and section~1 of~\cite{kenJFA02} for more discussion.
A third approach is to speak only of ``the conjectured free dimension'' which would mean that we conjecture
that the free group factors
are nonisomorphic, (in which case the free dimension has meaning);
however, we are not ready to conjecture nonisomorphism of free group factors.
In this article, we will use the notation $\fdim(M)=x$ as shorthand for (but without
actually phrasing) the more cumbersome ``has a
generating set of free  dimension $x$.''

In \cite{kenDuke} the concept of a \emph{standard embedding} from one interpolated free group factor into another is introduced.  We will rely heavily on this concept, in particular the following properties, shown in \cite{kenDuke},
\begin{enumerate}
\item For $A=L(F_{s})$ and $B=L(F_{s'})$, $s<s'$, then for $\phi:A\to B$ and projection $p\in A$, $\phi$ is standard if and only if $\phi|_{pAp}\to\phi(p)B\phi(p)$ is standard.
\item The inclusion $A\to A*B$ is standard if $A$ is an interpolated free group factor and $B$ is an interpolated free group factor, $L(\bbz)$, or a finite dimensional algebra other than $\bbc$.
\item The composition of standard embeddings is standard.
\item For $A_{n}=L(F_{s_{n}})$, with $s_{n}<s_{n'}$ if $n<n'$, and $\phi_{n}:A_{n}\to A_{n+1}$ a sequence of standard embeddings, then the inductive limit of the $A_{n}$ with the inclusions $\phi_{n}$ is $L_{F_{s}}$ where $s=\lim_{n\to\infty}s_{n}$.
\end{enumerate}

\section{Useful Lemmas}

\begin{lemma}\label{chainlemma}  Let $\mcm$ be a hyperfinite von Neumann algebra with separable predual and $D$ be a finite dimensional abelian subalgebra of it.  Then there exists a chain of finite dimensional subalgebras in $\mcm$ containing $D$ whose union is dense in $\mcm$.
\end{lemma}
\begin{proof}

$\mcm$ can be split as the direct sum of  a type I part and a type II part, each of which can be approximated separately, so we can deal with these cases individually.

If \mcms is type I then write it as $\bigoplus_{i\in I}\overset{p_{i}}{M_{m_{i}}\otimes L^{\infty}(X, \mu_{i})}$, for some finite measures $\mu_{i}$ and countable set $I$.  We may identify each $M_{m_{i}}\otimes L^{\infty}(X,\mu_{i})$ with $M_{m_{i}}(L^{\infty}(X,\mu_{i})$ in such a way that all elements of $D$ are identified with diagonal matrices.  Write $D=\oplus_{k\in I_{D}}\overset{p_{k}^{D}}{\underset{t_{k}^{D}}{\bbc}}$ with some finite index set $I_{D}$.

Set $A_{0 }=D$, and let $I_{n}\subseteq I$ consist of the first $n$ elements of $I$ (if $I$ is finite then $I_{n}=I$ for all $n>|I|$).  Set $P_{0,i}=\{X\}$ for all $i\in I$ and inductively choose $P_{n,i}$ to be a partition of $X$ so that: a) the measure $\mu_{i}$ of every set in $P_{n,i}$ is less than $1/2^{n}$ or is composed of a single atom, b) the partition $P_{n,i}$ is compatible with $D$ for $M_{m_{i}}$, and c) $P_{n,i}$ refines $P_{n-1,i}$.  By $P_{n,i}$ compatible with $M_{n_{i}}$ we mean that for each $e_{jj}\in M_{m_{i}}$ and $S_{m}\in P_{n,i}$, then $\chi_{S_{m}}e_{jj}\leq p_{k}^{D}$ for some $k\in I_{D}$.  Then set $A_{n}$ to be 
\[
\left(\bigoplus_{i\in I_{n}}\overset{p_{i}}{M_{m_{i}}\otimes \ell^{\infty}(P_{n,i},\mu_{i})}\right)\oplus \bbc^{k_{n}}.
\]
In the above the $\bbc^{k_{n}}$ is the span of $\{p_{k}^{D}-p_{k}^{D}\left(\sum_{i\in I_{n}}p_{i}\right),k\in I_{D}\}$. So we have constructed $D=A_{0}\subseteq A_{1}\subseteq \dots$ so that $\cup A_{i}$ is dense in \mcm.

For the type II case we assume $D=\bbc\oplus\bbc$, and from there the general case can be done inductively.  Let $p_{1}$ and $p_{2}$ be the minimal projections in $D$, $p_{1}+p_{2}=1$.  We write $\mcm=L^{\infty}(X,\mu)\otimes R$ where $R$ is the hyperfinite $II_{1}$ factor and $\mu$ is a finite measure.  We choose a representation of $R$ as the closure of $\cup_{k} M_{2^{k}}$, and denote the standard basis elements of $M_{2^{k}}$ as $e_{i,j}^{(2^{k})}$.  Use the inclusion of $M_{2^{k}}$ into $M_{2^{k+1}}$ which takes $e_{i,j}^{(2^{k})}$ to $e_{2i-1,2j-1}^{(2^{k+1})}+e_{2i,2j}^{(2^{k+1})}$.  

We know that projections are equivalent exactly when they have the same centre-valued trace.  Let $f\in L^{\infty}(\mu)$ be the centre-valued trace of $p_{1}$.  Note $f$ takes values only in $[0,1]$.   We can construct a projection which has centre-valued trace $f$ in the following way:
\[
p_{1}'=\sum_{k=1}^{\infty}\sum_{i=1}^{2^{k-1}}e_{(2i-1),(2i-1)}^{(2^{k})}\otimes\chi_{f^{-1}\left(\left(\frac{2i-1}{2^{k}},\frac{2i}{2^{k}}\right]\right)}
\]
where $\chi$ is the characteristic function.  To see how that works, note that any of the partial sums:
\[
p'_{1,N}=\sum_{k=1}^{N}\sum_{n=1}^{2^{k-1}}e_{(2i-1)(2i-1)}^{(2^{k})}\otimes\chi_{f^{-1}\left(\left(\frac{2i-1}{2^{k}},\frac{2i}{2^{k}}\right]\right)}
\]
is a member of $L^{\infty}\otimes M_{2^{N}}$.  Considering $p_{1,N}'$ as an $M_{2^{N}}$ valued function on $X$, note that in each region where $f$ takes values between $i/2^{N}$ and $(i+1)/2^{N}$, $p_{1,N}'$ is the matrix with 1s down the diagonal for the first $i$ entries, and zeros for the rest.  This means the centre-valued trace of $p_{1,N}'$ is the simple function $f_{N}$ taking only values of the form $i/2^{N}$ which approximates $f$ from below.  Thus $\|f_{N}-f\|\leq 1/2^{N}$, and so $p_{1}'$ has the desired centre-valued trace, and is thus equivalent to $p_{1}$.  So without loss of generality, by modifying our choice of dense subalgebra we can assume $p_{1}$ is of this form.

Start by defining the partition $P_{1}$ of $X$ as $P_{1}=\{S_{1,1},S_{1,2}\}$ where  $S_{1,1}=f^{-1}([0,1/2])$ and $S_{1,2}=f^{-1}((1/2,1])$, and define $\phi_{1,n}$ to be the characteristic function of $S_{1,n}$ times the identity.  Next define $A_{1}$ to be the algebra generated by $\{e_{1,1}^{(2)}\phi_{1,2},e_{2,2}^{(2)}\phi_{1,1},p_{1}-e_{1,1}^{(2)}\phi_{1,2},p_{2}-e_{2,2}^{(2)}\phi_{1,1}\}$.  Note these are four orthogonal projections, so they generate $\bbc^{4}$ which is clearly finite dimensional.  It is also clear that $p_{1}$ and $1-p_{1}$ are in this algebra, thus it contains $D$.

Next inductively construct partitions $P_{k}=\{S_{k,1},\dots, S_{k,\ell}\}$, satisfying the following conditions:  a)  For any $S_{k,n}\in P_{k}$ either its measure is less than $1/2^{k}$ or it is composed of a single atom, b) the range of $f$ on each $S_{k,n}\in P_{k}$ is contained in $(i/2^{k},(i+1)/2^{k}]$ for some integer $i$, and c) $P_{k}$ refines $P_{k-1}$

As before, we define $\phi_{k,n}$ to be the characteristic function of $S_{k,n}$. Define a function $r_{k}$ on $P_{k}$, so that $r_{k}(S_{k,n})=i$ where $f(S_{k,n})\subseteq(i/2^{k},(i+1)/2^{k}]$.   We set $A_{k}$ to be the algebra generated by: 
\begin{multline*}
\left\{e_{i,j}^{(2^{k})}\phi_{k,n}, \left(p_{1}\phi_{k,n}-\sum_{m=1}^{r_{k}(S_{k,n})}e_{m,m}^{(2^{k})}\phi_{k,n}\right),\right.\\
\left.\left(p_{2}\phi_{k,n}-\sum_{m=r_{k}(S_{k,n})+2}^{2^{k}}e_{m,m}^{(2^{k})}\phi_{k,n}\right)
|S_{k,n}\in P_{k,n},i\ne r_{k}(S_{k,n}),j\ne r_{k}(S_{k,n})\right\}.
\end{multline*}

By our representation of $p_{1}$, for $S_{k,n}\in P_{k}$ if  $i<r_{k}(S_{k,n})+1$ then  $e_{i,i}^{(2^{k})}\phi_{k,n}\leq p_{1}$ and if $i>r_{k}(S_{k,n})+1$, $e_{i,i}^{(2^{k})}\phi_{k+1,n}\leq p_{2}$, thus the algebra we get is $\bigoplus_{S_{k,n}\in P_{k}}\left(M_{2^{k}-1}\oplus \bbc\oplus\bbc\right)$.

Since $p_{1}\phi_{k,n}-\sum_{i=1}^{r_{k}(S_{k,n})}e_{i,i}^{(2^{k})}\phi_{k,n}$ is in $A_{k}$, and each of the elements of the sum is also in $A_{k}$, $p_{1}\phi_{k,n}\in A_{k}$ for each $S_{k,n}\in P_{k}$, and thus $p_{1}$ is, and by the same argument so is $p_{2}$ thus $D\subseteq A_{k}$.

To check that $A_{k-1}$ is a subalgebra of $A_{k}$, take any of the $e_{i,j}^{(2^{k-1})}\phi_{k-1,n}$ with $i,j\ne r_{k-1}(S_{k-1,n})$, then
\[ 
\sum_{n',S_{k,n'}\subseteq S_{k-1,n}}\left(e_{2i-1,2j-1}^{(2^{k})}\phi_{k,n'}+e_{2i,2j}^{(2^{k})}\phi_{k,n'}\right)=e_{i,j}^{(2^{k-1})}\phi_{k-1,n}.
\]
For $p_{1}\phi_{k-1,n}-\sum_{i=1}^{r_{k-1}(S_{k-1,n})}e_{i,i}^{(2^{k-1})}\phi_{k-1,n}$, we have already established that $p_{1}\phi_{k,n'}$ for any $n'$ is in $A_{k}$ and so since $p_{1}\phi_{k-1,n}=\sum_{n',S_{k,n'}\subseteq S_{k-1,n}}p_{1}\phi_{k,n'}$ it is in $A_{k}$.  Since all the terms in the sum are also in $A_{k}$, then  $p_{1}\phi_{k-1,n}-\sum_{i=1}^{r_{k-1}(S_{k-1,n})}e_{i,i}^{(2^{k})}\phi_{k-1,n}$ must be too.  Thus $A_{k-1}\subseteq A_{k}$.

Furthermore we see this is dense in $\mcm$ since our restriction on the measure of $S_{k,n}$ ensures ever greater refinement in our approximation of $L^{\infty}[\mu]$, and $R$ is approximated by $M_{2^{k+1}}$ with only one row and one column missing, whose contribution goes to zero as $k$ goes to infinity.

\end{proof}

Notation:  We denote $\Lambda^{0}(X,Y)$ to be the set of all non-empty alternating words of elements of $X$ and $Y$.  For an algebra $X$ with a trace, we use $X^{0}$ to denote the traceless elements of $X$ and if it has an expectation onto $D$, we use $X^{00}$ to denote the expectationless elements of $X$.

\begin{lemma}\label{gluelemma}

Let $\mcn=(M_{m}\oplus M_{n-m}\oplus B)*_{D}C$ and $\mcm=(M_{n}\oplus B)*_{D}C$, where $B$, $C$ are von Neumann algebras with finite trace and $D$ is finite dimensional and abelian.  We embed $\mcn$ in $\mcm$ by mapping $M_{m}$ and $M_{n-m}$ as blocks on the diagonal of $M_{n}$, and $B$ and $C$ by the identity.  Assume there exists a factor $\mathcal{F}$ with $M_{m}\oplus M_{n-m}\subseteq \mathcal{F}\subseteq \mcn$. Let $p_{i}^{D}$ be a minimal projection in $D$.  Then for any minimal projection $p\in M_{m}$ such that $p\leq p_{i}^{D}$, $p\mcn p *L(\bbz)\cong p\mcm p$

\end{lemma}
\begin{proof}  

Denote $A=M_{m}\oplus M_{n-m}\oplus B$ and $A'=M_{n}\oplus B$.  We represent the matrices in such a way that $D$ embeds diagonally and $p=e_{11}$

Note $\mcm=vN(\mcn\cup {e_{1n}})$.  Now since $\mathcal{F}$ is a factor we can find a partial isometry $v\in\mathcal{F}\subseteq\mcn$ such that $v^{*}v=e_{11}$ and $vv^{*}=e_{nn}$.  Let $u=v^{*}e_{n1}$, and thus $uu^{*}=u^{*}u=e_{11}$.

Since $u$ and $\mcn$ generate $\mcm$ we know that $u$ and $e_{11}\mcn e_{11}$ generate $e_{11}\mcm e_{11}$.  We claim that $u$ and $e_{11}\mcn e_{11}$ are $*-$free with respect to the trace.  We must show that any element of $\Lambda^{0}(\{u^{k}|k\in\bbz\backslash \{0\}\},(e_{11}\mcn e_{11})^{0})$ has zero trace.

By breaking up $u$ into $v^{*}e_{n1}$, and noting $v\in \mcn$, and $v=e_{nn}ve_{11}$ we see that 
\[
\Lambda^{0}(\{u^{k}|k\in\bbz\backslash\{0\}\},(e_{11}\mcn e_{11})^{0})\subseteq \Lambda^{0}(\{e_{1n},e_{n1}\},\{(e_{11}\mcn e_{11})^{0},e_{11}\mcn e_{nn},e_{nn}\mcn e_{11},(e_{nn}\mcn e_{nn})^{0}\}).
\]

We can write $x-E_{D}^{\mcn}(x)$ for any $x\in\mcn$ as the SOT limit of a bounded sequence in the span of $\Lambda^{0}(A^{00},C^{00})$.  Then since  non-zero elements of $e_{11}Ae_{11}$ or $e_{nn}Ae_{nn}$ have non-zero trace, we can write elements of $(e_{11}\mcn e_{11})^{0}$ and $(e_{nn}\mcn e_{nn})^{0}$ as SOT limits of bounded sequences in the span of $\Lambda^{0}(A^{00},C^{00})\backslash A^{00}$ (thus the expectation is zero also).  Similarly note that since $e_{11}Ae_{nn}=0=e_{nn}Ae_{11}$, elements of $e_{11}\mcn e_{nn}$ and $e_{nn}\mcn e_{11}$ can be written in the same way.

Thus we can approximate any element of 
\[
\Lambda^{0}(\{e_{1n},e_{n1}\},\{(e_{11}\mcn e_{11})^{0},e_{11}\mcn e_{nn},e_{nn}\mcn e_{11},(e_{nn}\mcn e_{nn})^{0}\})
\] by bounded sequences in the span of $\Lambda^{0}(\{e_{n1},e_{1n}\},\spn(\Lambda^{0}(A^{00},C^{00})\backslash A^{00}))$.  Now note that $E_{D}(Ae_{n1}A)=E_{D}(Ae_{1n}A)=0$.  Thus 
\[
\Lambda^{0}(\{e_{n1},e_{1n}\},\spn(\Lambda^{0}(A^{00},C^{00})\backslash A))\subseteq \spn( \Lambda^{0}(A'^{00},C^{00})).
\]
By freeness, elements of $\Lambda^{0}(A^{00},C^{00})$ have expectation zero, and thus trace zero, and so $u$ and $e_{11}\mcn e_{11}$ are $*$-free and $u$ is a Haar unitary.

Thus $e_{11}\mcn e_{11}*L(Z)\cong e_{11}\mcm e_{11}$.

\end{proof}

\begin{lemma}\label{tensorlemma1}  Let $\mcm=((M_{n}\otimes A)\oplus B)*_{D}C$ and $\mcn=(M_{n}\oplus B)*_{D}C$ for $A,B$, and $C$ von Neumann algebras with finite trace, and $D$ a finite dimensional abelian von Neumann algebra, where $C,B, M_{n}$ have expectations onto $D$ and where $E_{D}^{M_{n}\otimes A}=E_{D}^{M_{n}}\otimes \tau_{A}$.  Let $p$ be a minimal projection in $M_{n}$ which sits under a minimal projection in $D$.  Then $p\mcn p*A\cong p\mcm p$
\end{lemma}
\begin{proof}
Since $pA$ and $\mcn$ generate $\mcm$ we see that $p\mcn p$ and $A$ embedded as $pA$ generate $p\mcm p$.  Thus we need only show that these two algebras are $*$-free.

As in the previous proof, we see that traceless elements of $p\mcn p$ are expectationless and traceless elements of $p(M_{n}\oplus B)p$ are zero, so elements of $p\mcn p$ can be approximated by $\spn(\Lambda^{0}((M_{n}\oplus B)^{00},C^{00})\backslash(M_{n}\oplus B)^{00})$.  It is easy to see that traceless elements of $pA$ are also expectationless.  

Any traceless element of $pA$ multiplied on both left and right by elements of $M_{n}$ is expectationless (since each entry has zero trace in $A$).  Thus any element of $\Lambda^{0}((pA)^{0},(p\mcn p)^{0})$ can be approximated by words in the span of $\Lambda^{0}(((M_{n}\otimes A)\oplus B)^{00},C^{00})$, and thus expectationless and traceless, and thus $pA$ and $p\mcm p$ are $*-$free.
\end{proof}

\begin{defn}  We call an embedding $\phi:\mcm\to\mcn$ a \emph{simple step} if it follows one of the two following patterns:
\begin{enumerate}
\item $\mcn=\overset{p_{1}}{A}\oplus \overset{q}{B}$, $\mcm=\overset{p_{2}}{A}\oplus \overset{p_{3}}{A}\oplus \overset{q}{B}$, with $p_{1}=p_{2}+p_{3}$, and $\phi(a,b)=(a,a,b)$.
\item $\mcn=\underset{t}{M_{n}}\oplus \underset{t}{M_{m}}\oplus \overset{p}{B}$, $\mcm=\underset{t}{M_{n+m}}\oplus\overset{p}{B}$ where $\phi(x,y,b)=\left(\left[\begin{matrix} x & 0 \\ 0 & y\\\end{matrix}\right],b\right)$.
\end{enumerate}
\end{defn}

\begin{lemma}\label{decomplemma}  If we have two finite dimensional von Neumann algebras $\mcn$ and $\mcm$ and a trace preserving embedding $\phi:\mcn\to\mcm$, then $\phi$ can be written a the composition of as sequence of simple steps.

\end{lemma}
\begin{proof}
Since $\mcn$ and $\mcm$ are both finite dimensional we can write them as direct sums of matrix algebras, $\mcn=\overset{p_{1}}{M_{n_{1}}}\oplus\dots\oplus \overset{p_{s}}{M_{n_{s}}}$ and $\mcm=\overset{q_{1}}{M_{m_{1}}}\oplus\dots\oplus \overset{q_{t}}{M_{m_{t}}}$.  Then consider the Bratteli diagram for this embedding, $G=(V,E)$, $V=\{p_{i},q_{j}|1\leq i\leq s,1\leq j\leq t\}$.  For each edge linking $p_{i}$ to $q_{j}$ we decorate it as usual with the partial multiplicity of the embedding, $a_{i,k}$, in addition we decorate it with a number $\alpha_{i,j}$, which is the trace of $p_{i}q_{j}$.

Then to break this map into simple steps, we start by making copies of the $M_{n_{j}}$ by using the first type of simple step, so that for each edge $e_{i,j}$ there are $a_{i,j}$ copies of $M_{n_{j}}$, each with trace $\alpha_{i,j}/a_{i,j}$.  This leaves us with a multimatrix algebra, where each matrix algebra is associated with an edge $e_{i,j}$ (which is associated to $a_{i,j}$ matrix algebras).  Then for every $1\leq j\leq t$, we use the second type of simple step to join together all the matrix algebras associated to edges ending in $j$.

\end{proof}

\begin{defn}\label{graphdef}  We define the graph $G_{D}^{A}$ of a hyperfinite von Neumann Algebra $A$ and multimatrix subalgebra $D=\oplus_{i\in I_{D}}\overset{p_{i}^{D}}{M_{n_{i}}}$ as follows.  Our vertex set is $I_{D}$.  Vertices $i,j\in I_{D}$ are connected by an edge if $p_{j}Ap_{i}\ne 0$.

We also use $G_{D}^{A,B}$ to denote the union of the graphs $G_{D}^{A}$ and $G_{D}^{B}$, where $B$ is another von Neumann algebra with $D$ as a subalgebra.  

\end{defn}
Remark:  Note the graph as defined for multimatrix algebras in \cite{kenAmJM} is connected if and only if the graph in Definition \ref{graphdef} is connected, however it is not the same graph.  

\section{Main Results}

\begin{propn}\label{mainprop}
Let $A$ and $B$ be hyperfinite von Neumann algebras with finite trace, and let $D$ be a finite dimensional subalgebra of both, so that the graph $G_{D}^{A,B}$ is connected, and so that no minimal projection in $D$ is also minimal in $A$ or $B$.  Then $A*_{D}B=F\oplus_{r=1}^{R}M_{n_{r}}$ where $F$ is either an interpolated free group factor or a diffuse type I hyperfinite algebra.  Furthermore $\fdim(A*_{D}B)=\fdim(A)+\fdim(B)-\fdim(D)$.
\end{propn}
\begin{proof}  
The proof of Lemma 5.2 from \cite{kenAmJM} shows that we may assume without loss of generality that $D$ is abelian.

Write $A=\overset{p_{d}^{A}}{ A_{d}}\oplus \overset{p_{a}^{A}}{A_{a}}$ and $B=\overset{p_{d}^{B}}{B_{d}}\oplus\overset{p_{a}^{B}}{B_{a}}$, where $A_{a}$ and $B_{a}$ are the type I   atomic parts, and $A_{d}$ and $B_{d}$ the diffuse parts.  Write $D=\oplus_{k\in I_{D}}\overset{p_{k}^{D}}{\underset{t^{D}_{k}}{\bbc}}$, $A_{a}=\oplus_{\ell\in I_{A_{a}}}\underset{t_{\ell}}{M_{n_{\ell}}}$, and $B_{a}=\oplus_{\ell\in I_{B_{a}}}\underset{t_{\ell}}{M_{n_{\ell}}}$, with $I_{D}$, $I_{A_{a}}$, and $I_{B_{a}}$ disjoint.

Now using Lemma \ref{chainlemma} we can construct a sequence $A_{d}(i)$ of finite dimensional subalgebras of $A_{d}$ containing $p_{d}^{A}Dp_{d}^{A}$ and whose inductive limit is $A_{d}$.  Denote $A_{d}(i)=\oplus _{\ell\in I_{A_{d}(i)}}\overset{p_{\ell}}{\underset{t_{\ell}}{M_{n_{\ell}}}}$, and $B_{d}(j)=\oplus _{\ell\in I_{B_{d}(i)}}\overset{p_{\ell}}{\underset{t_{\ell}}{M_{n_{\ell}}}}$.  Note $I_{D},I_{A_{d}(i)},$ and $I_{B_{d}(i)},$ are finite, but $I_{A_{a}}$ and $I_{B_{a}}$ may not be. 

Also, since $A_{d}$ has no minimal projections, we can choose the $A_{d}(i)$ in such a way that the trace of the any minimal projection $p\in A_{d}(i)$, $p\leq p_{k}^{D}$ is less than $t_{k}^{D}-\tau(p')$ for any minimal projection  $p'\in B$, $p'\leq p_{k}^{D}$.  Note we can do this, since we know the minimal projections $p'\in B, p'\leq p_{k}^{D}$ have traces adding to $t_{k}^{D}$, and they must be strictly less than $t_{k}^{D}$.  We can use Lemma \ref{decomplemma} to make sure the steps between $A_{d}(i)$ and $A_{d}(i+1)$ are all simple steps.  We construct $B_{d}(j)$ in the same way.  Denote $A(i)=A_{d}(i)\oplus A_{a}$ and $B(j)=B_{d}(j)\oplus B_{a}$.  Note $A(i)=\oplus_{\ell\in I_{A(i)}}\underset{t_{\ell}}{M_{n_{\ell}}}$ where $I_{A(i)}=I_{A_{d}(i)}\cup I_{A_{a}}$.  Start our sequences far enough along that $G_{D}^{A(i),B(j)}=G_{D}^{A,B}$.

First, we claim that for any $i,j$ 
\[
M(i,j)=A(i)*_{D}B(j)=\overset{p_{0}^{M}}{F_{i,j}}\oplus\bigoplus_{r=1}^{R}\overset{p^{M}_{r}}{M_{n_{r}}},
\]
where $F_{i,j}$ is either an interpolated free group factor or a diffuse type I hyperfinite algebra.  We will also show that only the $F_{i,j}$  depends on $i$ and $j$.  Then we shall show that $F_{i,j}\to F_{i+1,j}$ is standard.

We begin by examining the construction in Theorem 5.1 in \cite{kenAmJM}.  For this we construct $N_{i,j}(k)=vN(p_{k}^{D}A(i)p_{k}^{D}\cup p_{k}^{D}B(j)p_{k}^{D})$ for each $k$.  This is of the form $F_{k}\oplus\bigoplus_{y\in Y(k)}M_{n_{y}}(\bbc)$, where $F_{k}$ is either an interpolated free group factor or a diffuse type I hyperfinite algebra.  Here each element of $y\in Y(k)$ corresponds to a pair $(\ell,\ell')$, $\ell\in I_{A(i)},\ell'\in I_{B(j)}$, representing a pair of matrix algebras in $A(i)$ and $B(j)$, so that $t_{\ell}/\lambda_{\ell,k}+t_{\ell'}^{B}/\lambda_{\ell',k}>t_{k}^{D}$, where $\lambda_{\ell,k}$ is the partial multiplicity of $p_{k}^{D}$ in $M_{n_{\ell}}$ (i.e. $M_{\lambda_{\ell,k}}\cong p_{k}^{D}M_{n_{\ell}}p_{k}^{D}$), and the same for $\lambda_{\ell',k}$.  Importantly, because of our choice of $A_{d}(i)$ and $B_{d}(j)$, if $t_{\ell}/\lambda_{\ell,k}+t_{\ell'}^{B}/\lambda_{\ell',k}>t_{k}^{D}$ then $\ell\in I_{A_{a}}$ and $\ell'\in I_{B_{a}}$.

Define a \emph{connector} $v$ to be a partial isometry in  $A(i)$ (or $B(j)$) so that $vv^{*}$ and $v^{*}v$ are minimal in $A(i)$ (or $B(j)$) and so $vv^{*}\leq p_{k}^{D}$ and $v^{*}v\leq p_{k'}^{D}$ for some $k,k'\in I_{D}$.

All further construction affecting the matrix algebras of $M(i,j)$ is based on those $Y(k)$s and connectors between them.  Thus, these are entirely determined by the $A_{a}$ and $B_{a}$, and thus remain unchanged from $M(i,j)$ to $M(i+1,j)$.

Now we show that the various summand factors $F_{k}$ are linked by connectors.  To do this we show that for any connector $v$ in $A(i)$ or $B(j)$ with $vv^{*}\leq p_{k}^{D}$, $vv^{*}$ is not orthogonal to $F_{k}$.  Without loss of generality, assume $v\in A(i)$.  Then $vv^{*}\in N_{i,j}(k)$, and is a minimal projection in $p_{k}^{D}A(i)p_{k}^{D}$.  Using Theorem 3.2 in \cite{kenAmJM}, we can determine the traces of the matrix factor summands of $N_{i,j}(k)$ which are not orthogonal to $vv^{*}$.  Let $t=\tau(vv^{*})$, and $\lambda$ the partial multiplicity of $p_{k}^{D}$ in the factor summand of $A(i)$ containing $vv^{*}$.  Then the matrix factor summands of $N_{i,j}(k)$ which are not orthogonal to $vv^{*}$ correspond to those factor summands $\underset{t_{\ell,k}}{M_{\lambda_{\ell,k}}}$ of $p_{k}^{D}B(j)p_{k}^{D}$ with $t/\lambda+t_{\ell}/\lambda_{\ell,k}>t_{k}^{D}$.  Note this can only happen if at least one of $\lambda$ or $\lambda_{\ell,k}$ equals 1.  Let $L\subseteq I_{B(j)}$ be the set of $\ell$ satisfying this inequality.  It is easy to check that $L$ is finite.

First let us assume $\lambda=1$.  Then for each $\ell\in L$ we have a matrix algebra $\underset{\lambda_{\ell,k}(t+\frac{t_{\ell}}{\lambda_{\ell,k}}-t_{k}^{D})}{M_{\lambda_{\ell,k}}}$.  Thus the total trace of these is
\[
\sum_{\ell\in L}\lambda_{\ell,k}^{2}(t+\frac{t_{\ell}}{\lambda_{\ell,k}}-t_{k}^{D})
\leq t\left(\sum_{\ell\in L}\lambda_{\ell,k}^{2}\right)+t_{k}^{D}-t_{k}^{D}\left(\sum_{\ell\in L}\lambda_{\ell,k}^{2}\right)
=t-(t_{k}^{D}-t)(\sum_{\ell\in L}\lambda_{\ell,k}^{2}-1),
\]
since $\lambda_{\ell,k}t_{\lambda_{\ell,k}}=\tau(p_{k}^{D}I_{M_{\ell}}p_{k}^{D})$, and $p_{k}^{D}I_{M_{\ell}}p_{k}^{D}\leq p_{k}^{D}$.  Now since $t_{k}^{D}>t$, unless $|L|=1$ and $\lambda_{\ell,k}=1$ this is strictly less than $t$.  If $\lambda_{\ell,k}=1$, and $|L|=1$, the first inequality holds strictly, since otherwise $\tau(p_{k}^{D})=\tau(p_{k}^{D}I_{M_{\ell}}p_{k}^{D})$ and this is minimal in $B(j)$, thus $p_{k}^{D}$ is minimal in $B$, contradicting our assumption.  This shows the total trace of the matrix factor summands of $N_{i,j}(k)$ not orthogonal to $vv^{*}$ is strictly less than that of $vv^{*}$, and thus $F_{k}$ is not orthogonal to $vv^{*}$.

If instead $\lambda>1$ and thus $\lambda_{\ell,k}=1$ for each $\ell\in L$.  Let $N=|L|$, then for each such $\ell\in L$ we have a matrix algebra $\underset{\lambda(\frac{t}{\lambda}+t_{\ell}-t_{k}^{D})}{M_{\lambda}}$.  Thus the total trace of these is:
\[
\sum_{\ell\in L}\lambda^{2}(\frac{t}{\lambda}+t_{\ell}-t_{k}^{D})\leq tN\lambda+\lambda^{2}t_{k}^{D}-\lambda^{2}Nt_{k}^{D}=t\lambda-\lambda(N-1)(\lambda t_{k}-t),
\]
which is then strictly less than $t\lambda$ if $N>1$.  As before, if $N=1$, then the first inequality holds strictly, otherwise $p_{k}^{D}$ would be minimal in $B$.  In this case, since $vv^{*}$ is a minimal projection in $M_{\lambda}$ as a factor summand of $N_{i,j}(k)$,   we can form $p_{1},\dots, p_{\lambda}$, minimal projections in $N_{i,j}(k)$ which are all equivalent to $vv^{*}$ in $N_{i,j}(k)$ and mutually orthogonal.  Thus the trace of their sum is $t\lambda$, which we have established is strictly greater than the total trace of all the matrix factor summands not orthogonal to them, and so they must not be orthogonal to $F_{k}$, and so $vv^{*}$ is not either.

The construction of $M(i,j)$ in Theorem 5.1 in \cite{kenAmJM}, proceeds by taking $N_{i,j}=\oplus_{k\in I_{D}}N_{i,j}(k)$, and adjoining connectors.  Let $S$ denote the set of connectors added, and let $N_{i,j}(S')=vN(N_{i,j}\cup S')$, for any $S'\subseteq S$.  Thus $M(i,j)=N_{i,j}(S)$.  Let $S_{v}$ be the set of connectors added before $v\in S$.  From the above, for any $v\in S$ we know that neither $vv^{*}$ nor $v^{*}v$ are minimal in $N_{i,j}$, thus when move from $N_{i,j}(S_{v})$ to $N_{i,j}(S_{v}\cup \{v\})$, we apply Lemma 4.2 in \cite{kenAmJM}, to determine $\bar{v}N_{i,j}(S_{v}\cup \{v\})\bar{v}$ where $\bar{v}=vv^{*}+v^{*}v$.  This shows that it contains exactly one free group factor (it can't be a hyperfinite, since we know $\dim\bar{v}N_{i,j}(S_{v})\bar{v}=\infty>4$).  Thus for any $v\in S$,  $vv^{*}\leq p_{k}^{D}$ and $v^{*}v\leq p_{k'}^{D}$ for $k,k'\in I_{D}$, then $F_{k}$ and $F_{k'}$ are contained in some $F$ which is an interpolated free group factor summand of $M(i,j)$ (noting if $F_{k}$ is hyperfinite, then $p_{k}^{D}A(i)p_{k}^{D}=\overset{p_{1}}{\bbc}\oplus\overset{p_{k}^{D}-p_{1}}{\bbc}$ and $p_{k}^{D}B(i)p_{k}^{D}=\underset{p_{2}}{\bbc}\oplus\overset{p_{k}^{D}-p_{2}}{\bbc}$, and thus $vv^{*}\in \{p_{1},p_{2},1-p_{1},1-p_{2}\}$.  By Theorem 1.1 in \cite{kenDuke} all of these have central support which contains $I_{F_{k}}$).  Then since the graph $G_{D}^{A(i),B(j)}$ is connected, there is exactly one free group factor summand of $M(i,j)$, unless, $A=\bbc\oplus\bbc$, $B=\bbc\oplus\bbc$ and $D=\bbc$, in which case it is a type I diffuse hyperfinite algebra.

Next we show that the inclusion of $A_{d}(i)$ into $A_{d}(i+1)$ is standard.  Since it is a simple step it is one of two types.  First it  could make two copies of a matrix algebra $M_{n_{\ell}}$ in $A_{d}(i)$ contained in the free group factor summand $F_{i,j}$ in $M(i,j)$.  Lemma \ref{tensorlemma1} shows that $F_{i,j}\to F_{i+1,j}$ is a standard embedding.

Alternately it could be the addition of a connector, connecting $M_{n_{\ell}},M_{n_{\ell'}}\in A_{d}(i)$, giving us $M_{\ell''}\in A_{d}(i+1)$.  Since $M_{n_{\ell}}$ and $M_{n_{\ell'}}$ are contained in $F_{i,j}$, we can apply Lemma \ref{gluelemma} to show $F_{i,j}\to F_{i+1,j}$ is a standard embedding.

By Theorem 5.1 in \cite{kenAmJM} we know that at each stage the free dimension of $M(i,j)$ is the sum of the free dimensions of $A(i)$ and $B(j)$ minus that of $D$.  By Proposition 4.3 in \cite{kenDuke} we know that since the embeddings are standard, the free dimension of the inductive limit of the $M(i,j)$ is the limit of the free dimensions of the $M(i,j)$.  Thus the inductive limit, $A*_{D}B$ has free dimension of the sum equal to that of $A$ and $B$ minus that of $D$. 

Note that since the matrix portion remain constant and thus can be computed by computing an earlier stage in the chain, using Theorem 5.1 in \cite{kenAmJM}.
\end{proof}

\begin{theorem}\label{maintheorem}  Let $A$ and $B$ be hyperfinite von Neumann Algebras, and $D$ be a finite dimensional subalgebra of both, with $G^{A,B}_{D}$ connected.  Then $A*_{D}B$ is the direct sum of a finite number of interpolated free group factors and possibly a hyperfinite von Neumann algebra.  Furthermore $\fdim(A*_{D}B)=\fdim(A)+\fdim(B)-\fdim(D)$.
\end{theorem}
\begin{proof}  As before we assume $D$ is abelian and set $A=A_{a}\oplus A_{d}$ and $B=B_{a}\oplus B_{d}$.  In choosing our sequence the only difference from the situation in Proposition \ref{mainprop} is that when we ensure the minimal projections $p\in A_{d}(i)$, $p<p_{k}^{D}$ have smaller trace than $t_{k}^{D}-\tau(p')$, for all minimal projections $p'\in B$, $p'<p_{k}^{D}$, we are excluding the cases where $p'=p_{k}^{D}$ (in the context of prop \ref{mainprop} these did not exist).  In the case where $p_{k}^{D}$ is minimal in $B$ we instead require that $\tau(p)\leq t_{k}^{D}/4$ for all $p\in A,p\leq p_{k}^{D}$.

Unfortunately it is no longer true that we must have only one free group factor summand, nor that all of $A_{d}(i)$ and $B_{d}(j)$ must be contained in interpolated free group factors.  Nor is it true that all elements of $p_{k}^{D}A_{d}(i)p_{k}^{D}$ for some $k$ must be in the same interpolated free group factor.  

Instead for each free group factor summand $F$ in $M(i,j)$ which is not orthogonal to both $A_{d}$ and $B_{d}$, we shall associate a non-empty subset $Q_{F}\subseteq I_{D}$, so that $Q_{F}\cap Q_{F'}=\emptyset$ if $F\ne F'$.  Furthermore if $k\in Q_{F}$ the we will have $p_{k}^{D}F'p_{k}^{D}=0$ for all $F'\ne F$.

Each $k\in I_{D}$ can be assigned to a $Q_{F}$ in one of two ways, or not assigned at all.  \emph{Method I} is used if  $p_{k}^{D}$ is not minimal in either $A$ or $B$.  Then we have established that there is one interpolated free group factor summand $F$ of $M(i,j)$ which contains $p_{k}^{D}A_{d}(i)p_{k}^{D}$ and $p_{k}^{D}B_{d}(i)p_{k}^{D}$, we assign $k$ to $Q_{F}$

Suppose on the contrary, for $k\in I_{D}$ that $p_{k}^{D}$ is minimal in  either $A$ or $B$, and without loss of generality assume it is $B$.  Additionally assume $p_{k}^{D}$ is not orthogonal to $A_{d}$.

 As in the previous proof, and the proof of Theorem 5.1 in \cite{kenAmJM}, we construct $M(i,j)$ with $N_{i,j}(k)=vN(p_{k}^{D}A(i)p_{k}^{D}\cup p_{k}^{D}B(i)p_{k}^{D})$, and then define $N_{i,j}=\oplus_{k\in I_{D}}N_{i,j}(k)$.  Set $N_{i,j}(S)=vN(N_{i,j}\cup S)$ where $S$ is a set of connectors.

Then there is a set $S\subseteq (A(i)\cup B(j))$ of connectors so that $M(i,j)=N_{i,j}(S)$.  If there is a connector $v$ such that either $v^{*}v$ or $vv^{*}$ equals $p_{k}^{D}$, then $v\in B_{a}$, since $p_{k}^{D}$ is minimal in $B$ and not in $A(i)$.  Without loss of generality we can assume that there is at most one $v\in S$ so that $vv^{*}$ or $v^{*}v$ is $p_{k}^{D}$ (without loss of generality assume $vv^{*}$), since if there is another $w$ so that $ww^{*}=p_{k}^{D}$ we can replace it with $v^{*}w$.

We use \emph{Method II} when there is such a $v$ and  $v^{*}v$ is not minimal and central in $\bar{q}N_{i,j}(S_{0})\bar{q}$ where $\bar{q}=vv^{*}+v^{*}v$ and $S_{0}=S\backslash \{v\}$.

First let us show that this does not depend on the choice of $v$.  Assume there exists $v$ and $v'$, where $vv^{*}=v'v'^{*}=p_{k}^{D}$.  Now assume $v^{*}v$ is minimal and central in $\bar{q}N_{i,j}(S_{0})\bar{q}$, but $v'^{*}v'$ is not minimal and central in  $\bar{q'}N_{i,j}(S_{0})\bar{q'}$, where $\bar{q}=vv^{*}+v^{*}v$, $\bar{q'}=v'v'^{*}+v'^{*}v'$, and $S_{0}$ as before.  Then there must be some $w\in N_{i,j}(S_{0})$ such that $ww^{*}=v^{*}v$ and $w^{*}w=v'^{*}v$.  Note since $v^{*}v$ is minimal and central in $\bar{q}N_{i,j}(S_{0})\bar{q}$, it is minimal and commutes with $p_{k}^{D}$ in $N_{i,j}(S_{0})$.  Thus it must be in some factor summand of $N_{i,j}(S_{0})$ which is a matrix algebra orthogonal to $p_{k}^{D}$.  But using $w$, we see $v'^{*}v'$ must be in this same factor summand, contradicting our assumption.

As in the proof of theorem 5.1 in \cite{kenAmJM} we see that $\bar{q}M(i,j)\bar{q}=vN(\bar{q}N_{i,j}(S_{0})\bar{q}\cup \{v\})$, and furthermore that this is isomorphic to $\bar{q}N_{i,j}(S_{0})\bar{q}*_{\bbc\oplus\bbc}M_{2}(\bbc)$, where $\bbc\oplus\bbc$ is spanned by $vv^{*}$ and $v^{*}v$ and $M_{2}(\bbc)$ is generated by $v$.

Then Lemma 4.2 from \cite{kenAmJM} shows $\bar{q}M(i,j)\bar{q}$ is isomorphic to an interpolated free group factor $F$ (by our assumption on the size of projections in $A(i)$ we know the this isn't hyperfinite or zero), possibly direct sum with some matrix algebras.  Then $F$ is contained in some free group factor summand $F'$ in $M(i,j)$, and we assign $k$ to $Q_{F'}$.  From this we see that any free group factor in $M(i,j)$ other than $F'$ is orthogonal to $\bar{q}$, and thus to $p_{k}^{D}$.

In fact, in this case if we know $v^{*}v$ is not minimal and central in $\bar{q}N_{i,j}(S_{0})\bar{q}$, then we know it is not minimal there.  If it were minimal, but not central then there would be some partial isometry $w$ also in $\bar{q}N_{i,j}(S_{0})\bar{q}$, $w\in A$ so that $w^{*}w\leq v^{*}v$, and $ww^{*}\leq vv^{*}$, in which case, since $v^{*}v$ is minimal, $w^{*}w=v^{*}v$, and thus (since the traces are the same) $ww^{*}=vv^{*}$, and so $p_{k}^{D}=ww^{*}$, and this is minimal in $A$, contradicting our assumption.

This means that that there must be some value $\delta>0$ so that $\tau(v^{*}v)-\tau(p)>\delta$ for all minimal projections $p\in \bar{q}N_{i,j}(S_{0})\bar{q}$, $p\leq v^{*}v$.  We may start our sequence $A_{d}(i)$ far enough along so that any minimal projection in $A_{d}(i)$ has trace less than $\delta$.  In this case we actually have $p_{k}^{D}A_{d}(i)p_{k}^{D}\subseteq F'$ where $k\in Q_{F'}$.  Start our sequence far enough along that this is true for all $k$ assigned with Method II.\bigskip

We have now assigned the all the elements of $I_{D}$ that we are going to.  We must now show that $Q_{F}\ne \emptyset$ for any factor summand $F$ in $M(i,j)$ which is not orthogonal to $A_{d}(i)$ and $B_{d}(j)$.

To do this, take any $\ell\in I_{A_{d}(i)}$ or  $I_{B_{d}(j)}$ (without loss of generality assume $I_{A_{d}(i)}$).  Let $K=\{k\in I_{D}| p_{k}^{D}M_{n_{\ell}}p_{k}^{D}\ne 0\}$.  If there is any $k\in K$ and factor summand $F$ in $M(i,j)$ such that $k\in Q_{F}$, then $M_{n_{\ell}}\subseteq F$.

Conversely assume there is no such $k\in K$, i.e. $k\notin Q_{F}$ for all $F$.  We will show that $M_{n_{\ell}}$ is then orthogonal to all free group factor summands of $M(i,j)$.

We know that for each $k\in K$, $p_{k}^{D}$ is minimal in $B$, otherwise it would have been assigned by Method I.  For each $k\in K$, where $p_{k}^{D}$ is not central in $B$, choose a connector $v_{k}\in B_{a}$ such that $v_{k}v_{k}^{*}=p_{k}^{D}$, and let $V$ be the set of these connectors.  Now there cannot be any $v_{k}\in V$ such that $v_{k}v_{k}^{*}=p_{k}^{D}$ and $v_{k}^{*}v_{k}=p_{k'}^{D}$ for $k,k'\in K$, because choosing a connector $w\in M_{n_{\ell}}\subseteq A_{d}(i)$ so that $w^{*}w\leq p_{k}^{D}$ and $ww^{*}\leq p_{k'}^{D}$, we would have $v_{k}^{*}v_{k}$ not central in $\bar{q}N_{i,j}(S_{0})\bar{q}$, meaning $k$ would have been assigned by Method II, contrary to the hypothesis.

Thus we can choose $S$ to be a set of connectors in $M(i,j)$ containing $V$, so that $N_{i,j}(S)=M(i,j)$ and such that for each $k\in K$ there is at most one $v\in S$ such that either $vv^{*}$ or $v^{*}v$ equals $p_{k}^{D}$.

Now then $M_{n_{\ell}}$ is a factor summand in $N_{i,j}(S\backslash V)$, since in this algebra every $p_{k}^{D}$ not orthogonal to $M_{n_{\ell}}$ is minimal and central in $B(j)$.  We now write $V=\{v_{k(1)},\dots,v_{k(r)}\}$.  We consider the chain of embeddings,
\[
N_{i,j}(S\backslash V)\to N_{i,j}(S\backslash V\cup \{v_{k(1)}\})\to\dots\to N_{i,J}(S)=M(i,j).
\]
In each of them the algebra $M_{n_{\ell}}$ is embedded as the corner of a factor summand which is also a matrix algebra.  Thus $M_{n_{\ell}}$ is orthogonal to all free group factor summands in $M(i,j)$.

If Method I or II is applied to some $k\in I_{D}$, in $M(i,j)$, it will also be applied in $M(i+1,j)$.  Furthermore if there is some factor summand $F$ in $M(i,j)$ so that $k,k'\in Q_{F}$, then there must be a partial isometry $v\in M(i,j)$ so that $vv^{*} \leq p_{k}^{D}$ and $v^{*}v\leq p_{k'}^{D}$, and both $vv^{*}$ and $v^{*}v$ are in either $A_{d}(i)$ or $B_{d}(j)$ (though $v$ may not be).  Then, since that $v$ is embedded in $M(i+1,j)$, there must be some factor summand $F'$ of $M(i+1,j)$ so that $k,k'\in Q_{F'}$.  Thus the $Q_{F}$ must be eventually stable, and thus so must be the number of free group factor summands in the sequence of $M(i,j)$.  Without loss of generality, by starting the sequence $A(i)$ and $B(j)$ far enough along, we assume the number of interpolated free group factor summands of $M(i,j)$ is constant in $i,j$, as is the family $\{Q_{F}\}_{F}$.\bigskip

Consider the embedding $M(i,j)\to M(i+1,j)$, which corresponds to applying a simple step to $A_{d}(i)\to A_{d}(i+1)$.  We will show that for any free group factor summand $F$ in $M(i,j)$, with central support $p_{F}$, the restriction of the embedding $F\to p_{F}M(i+1,j)p_{F}$ is either a standard embedding or the identity.

First assume it is the first type of simple step, where we make a copy of a matrix algebra $M_{n_{\ell}}$, $\ell\in I_{A_{d}(i)}$.  If there is a factor summand $F$ with $k\in Q_{F}$ so that $p_{k}^{D}M_{n_{\ell}}p_{k}^{D}\ne 0$, then we know $M_{n_{\ell}}\subseteq F$.  Then, as in the previous proposition, we can apply Lemma \ref{tensorlemma1} to see that this is a standard embedding.

Note that if there is no such $F$, then $M_{n_{\ell}}$ is the corner of a matrix algebra factor summand of $M(i,j)$, in which case we end up with two copies of this in $M(i+1,j)$.

Assume instead it is the second type of simple step, adding a connector between $M_{n_{\ell}}$ and $M_{n_{\ell'}}$, for some $\ell,\ell'\in I_{A_{d}(i)}$.  If there are $F,F'$ factor summands of $M(i,j)$ which contain $M_{n_{\ell}}$ and $M_{n_{\ell'}}$ respectively, then by our assumption $F=F'$, and as in the previous proposition we can apply Lemma \ref{gluelemma}, to show that $F\to p_{F}M(i+1,j)p_{F}$ is standard.

If neither $M_{n_{\ell}}$ nor $M_{n_{\ell'}}$ is contained in an interpolated free group factor summand of $M(i,j)$, then these are corners of factor summands $M_{m}$ and $M_{m'}$ in $M(i,j)$, which are embedded in a factor summand  $M_{m'+m}$ in $M(i+1,j)$.

Finally if only of of $M_{n_{\ell}}$ or $M_{n_{\ell'}}$ is contained in an interpolated free group factor summand, say $M_{n_{\ell}}$, then we can apply Lemma 4.2 from \cite{kenAmJM} again to see that these are embedded in a free group factor summand $F'$ of $M(i+1,j)$ so that $F\to p_{F}F'p_{F}$ is standard or the identity, where $p_{F}$ is the identity of $F$.

If $P_{M}(i,j)$ is the projection onto the hyperfinite and matrix portion of $M(i,j)$ we see that $P_{M}(i,j)\geq P_{M}(i',j')$ if $(i,j)\leq (i',j')$.  Taking the inductive limit of those, we get a hyperfinite algebra.

As before the free dimensions of each $M(i,j)$ add up correctly, so $\fdim(A*_{D}B)=\fdim(A)+\fdim(B)-\fdim(D)$.
\end{proof}
\begin{cor}  If $A$ and $B$ are hyperfinite von Neumann algebras with finite dimensional abelian von Neumann subalgebra $D$, so that $G_{D}^{A,B}$ is connected, and so that any minimal projection $p$ in $A$ or $B$, with $p\leq p_{k}^{D}$ a minimal projection in $D$ $\tau(p)<\frac{1}{2}\tau(p_{k}^{D})$, then $A*_{D}B=L(F_{s})$ $s=\fdim(A)+\fdim(B)-\fdim(D)$.
\end{cor}
\begin{proof}  Based on the condition $\tau(p)\leq \frac{1}{2}\tau(p_{k}^{D})$, we know that there is no matrix algebra component.  Furthermore, by the connectedness of $G_{D}^{A,B}$ there can be only one interpolated free group factor.
\end{proof}

\begin{defn}  Define $\mathcal{R}_{2}$ to be the set of finite von Neumann algebras which are the direct sum of a finite number of interpolated free group factors and a hyperfinite von Neumann algebra.  Note $\mathcal{R}_{2}$ strictly contains the set $\mathcal{R}$ from \cite{kenLMS}.
\end{defn}
\begin{theorem}\label{closedtheorem} For $A,B\in \mathcal{R}_{2}$ and $D$ a finite dimensional subalgebra of $A$ and $B$, $M=A*_{D}B$ is in $\mathcal{R}_{2}$, and if generating sets of $A$ and $B$ have free dimension $d_{A}$ and $d_{B}$, respectively, then $M$ has a generating set of free dimension $d_{A}+d_{B}-\fdim(D)$.
\end{theorem}  
\begin{proof}  
As in Theorem 4.4 in \cite{kenLMS}, we proceed by induction.  In this case on the total number of interpolated free group factor summands in $A$ and $B$.  The base case, where there are none, is Theorem \ref{maintheorem}.

From there we can proceed identically to the proof of Theorem 4.4 in \cite{kenLMS}.  Without loss of generality assume $A$ contains at least one interpolated free group factor summand, and let $p$ be the central support in $A$ of this factor summand.  Then by Lemma 4.3 in \cite{kenAmJM}, we know $pMp\cong pA*_{pD}(p\tilde{M}p)$, where $\tilde{M}=(pD\oplus(1-p)A)*_{D}B$.  Then we note that $(pD\oplus(1-p)A)$ is in $\mathcal{R}_{2}$ and has one fewer interpolated free group factor summand than $A$.  Thus, by our induction hypothesis $\tilde{M}\in\mathcal{R}_{2}$.  Then since $pA$ is an interpolated free group factor, we can apply Lemma 4.1 from \cite{kenLMS} to see that $pMp$ is an interpolated free group factor.  Then, as in Theorem 4.4 of \cite{kenLMS}, we note $M\cong\tilde{M}(1-C_{M}(p))\oplus(C_{M}(p)M)$, and that $\tilde{M}(1-C_{M}(p))\in\mathcal{R}_{2}$ and $(C_{M}(p)M)$ is an interpolated free group factor.  Thus $M\in\mathcal{R}_{2}$.

The free dimension calculation also proceeds exactly as in Theorem 4.4 of \cite{kenLMS}.

\end{proof}

\section{Examples}

\begin{eg}  For one of the simplest examples, consider 
\[
M=L^{\infty}(\mu_{1})\otimes R*_{\underset{\alpha}{\bbc}\oplus\underset{1-\alpha}{\bbc}}L^{\infty}(\mu_{2})\otimes R.
\]  
In this case, as long as the graph is connected, the embedding of $D$ does not matter, nor do the measures $\mu_{1}$ and $\mu_{2}$.  Since there are no minimal projections in $A$ or $B$, we know that $M$ is a single interpolated free group factor.  We also know $\fdim(M)=\fdim(A)+\fdim(B)-\fdim(D)=1+1-(1-\alpha^{2}-(1-\alpha)^{2})$,  Thus $M=L(F_{1+\alpha^{2}+(1-\alpha)^{2}})$.
\end{eg}

\begin{eg}  Next we look at an example where we get the direct sum of two free group factors, despite having connected graph.  Consider:
\[
A=\overset{p}{R}\oplus \underset{1/8}{M_{2}}\oplus \underset{1/8}{M_{2}}\oplus \overset{q}{R}
\]
\[
B=\underset{1/8}{\bbc}\oplus\underset{1/8}{M_{2}}\oplus\underset{1/4}{\bbc}\oplus\underset{1/8}{M_{2}}\oplus\underset{1/8}{\bbc}
\]
\[
D=\overset{p_{1}}{\underset{1/4}{\bbc}}\oplus\overset{p_{2}}{\underset{1/8}{\bbc}}\oplus\overset{p_{3}}{\underset{1/4}{\bbc}}\oplus\overset{p_{4}}{\underset{1/8}{\bbc}}\oplus\overset{p_{5}}{\underset{1/4}{\bbc}}
\]
Where $\tau(p)=\tau(q)=1/4$. 
\[
p_{1}=\left(
I_{R},
\left[\begin{matrix} 0 & 0\\0 & 0\\\end{matrix}\right],
\left[\begin{matrix} 0 & 0\\0 & 0\\\end{matrix}\right],
0_{R}\right)\in A, 
\left( 1,
\left[\begin{matrix} 1 & 0\\0 & 0\\\end{matrix}\right],
0,
\left[\begin{matrix} 0 & 0\\0 & 0\\\end{matrix}\right],
0\right)\in B
\]
\[
p_{2}=\left(
0_{R},
\left[\begin{matrix} 1 & 0\\0 & 0\\\end{matrix}\right],
\left[\begin{matrix} 0 & 0\\0 & 0\\\end{matrix}\right],
0_{R}\right)\in A, 
\left( 0,
\left[\begin{matrix} 0 & 0\\0 & 1\\\end{matrix}\right],
0,
\left[\begin{matrix} 0 & 0\\0 & 0\\\end{matrix}\right],
0\right)\in B
\]
\[
p_{3}=\left(
0_{R},
\left[\begin{matrix} 0 & 0\\0 & 1\\\end{matrix}\right],
\left[\begin{matrix} 1 & 0\\0 & 0\\\end{matrix}\right],
0_{R}\right)\in A, 
\left( 0,
\left[\begin{matrix} 0 & 0\\0 & 0\\\end{matrix}\right],
1,
\left[\begin{matrix} 0 & 0\\0 & 0\\\end{matrix}\right],
0\right)\in B
\]
\[
p_{4}=\left(
0_{R},
\left[\begin{matrix} 0 & 0\\0 & 0\\\end{matrix}\right],
\left[\begin{matrix} 0 & 0\\0 & 1\\\end{matrix}\right],
0_{R}\right)\in A, 
\left( 0,
\left[\begin{matrix} 0 & 0\\0 & 0\\\end{matrix}\right],
0,
\left[\begin{matrix} 1 & 0\\0 & 0\\\end{matrix}\right],
0\right)\in B
\]
\[
p_{5}=\left(
0_{R},
\left[\begin{matrix} 0 & 0\\0 & 0\\\end{matrix}\right],
\left[\begin{matrix} 0 & 0\\0 & 0\\\end{matrix}\right],
I_{R}\right)\in A, 
\left( 0,
\left[\begin{matrix} 0 & 0\\0 & 0\\\end{matrix}\right],
0,
\left[\begin{matrix} 0 & 0\\0 & 1\\\end{matrix}\right],
1\right)\in B
\]
We use $A(i)=M_{i}\oplus M_{2}\oplus M_{2}\oplus M_{i}$ and $B=B_{a}=B(j),\forall j$.  Applying theorem 5.1 from \cite{kenAmJM}, we see $M(i,j)=L(F_{\frac{9}{8}-\frac{1}{4i^{2}}})\oplus L(F_{\frac{9}{8}-\frac{1}{4i^{2}}})$.  Thus as $i\to\infty$, $M=A*_{D}B=L(F_{\frac{9}{8}})\oplus L(F_{\frac{9}{8}})$.  Checking, we see $\fdim(A)=\frac{31}{32}$, $\fdim(B)=\frac{7}{8}$ and $\fdim(D)=\frac{25}{32}$, this gives us that $\fdim(M)=\frac{17}{16}$ which matches our result.
\end{eg}
\begin{eg}  Next consider the case, $A=L^{\infty}[0,1]\oplus \underset{1/3}{M_{2}}$ and $B=\underset{1/3}{M_{2}}\oplus\underset{1/3}{\bbc}$, and $D=\overset{p_{1}}{\underset{1/3}{\bbc}}\oplus\overset{p_{2}}{\underset{1/3}{\bbc}}\oplus\overset{p_{3}}{\underset{1/3}{\bbc}}$.  $p_{1}=(I_{L^{\infty}},0)\in A$ and $(e_{11},0)\in B$, $p_{2}=(0,e_{11})\in A$ and $(e_{22},0)\in B$, and $p_{3}=(0,e_{22})\in A$, and $(0,1)\in B$.

Then we set $A(i)=\underset{1/(3i)}{\bbc^{i}}\oplus \underset{1/3}{M_{2}}$, and $B(j)=B\forall j$.  Then $M(i,j)=\underset{1/(3i)}{\bbc^{i}\otimes M_{3}}$.  Taking the inductive limit, this is $L^{\infty}\otimes M_{3}$.  

\end{eg}
\begin{eg}  Similarly, if $A=(L^{\infty}[0,1]\otimes R)\oplus \underset{1/3}{M_{2}}$ and $B=\underset{1/3}{M_{2}}\oplus\underset{1/3}{\bbc}$, and $D=\overset{p_{1}}{\underset{1/3}{\bbc}}\oplus\overset{p_{2}}{\underset{1/3}{\bbc}}\oplus\overset{p_{3}}{\underset{1/3}{\bbc}}$.  $p_{1}=(I_{L^{\infty}\otimes R},0)\in A$ and $(e_{11},0)\in B$, $p_{2}=(0,e_{11})\in A$ and $(e_{22},0)\in B$, and $p_{3}=(0,e_{22})\in A$, and $(0,1)\in B$.

Then we set $A(i)=\underset{1/(3i^{2})}{\bbc^{i}\otimes M_{i}}\oplus \underset{1/3}{M_{2}}$, and $B(j)=B\forall j$.  Then $M(i,j)=\underset{1/(3i^{2})}{\bbc^{i}\otimes M_{3i}}$.  Taking the inductive limit, this is $L^{\infty}\otimes R$.  

\end{eg}
\begin{eg}  Let $A=R$ and $B=\overset{p_{1}}{\underset{\alpha/2}{M_{2}}}\oplus \overset{p_{2}}{\underset{1-\alpha}{\bbc}}$ and $D=\overset{p_{1}}{\underset{\alpha}{\bbc}}\oplus\overset{p_{2}}{\underset{1-\alpha}{\bbc}}$, with $\alpha$ irrational.  Then we can set $A(i)=\underset{1/2^{i}}{M_{2^{i}-1}}\oplus \overset{q_{1}}{\underset{a}{\bbc}}\oplus \overset{q_{2}}{\underset{b}{\bbc}}$, where $a+b=1/2^{i}$, $q_{1}\leq p_{1}$ and $q_{2}\leq p_{2}$.  $B(j)=B$ for all $j$.  Then $M(i,j)=L(F_{s})\oplus\underset{b}{\bbc}$ for some $s$.  Noting $b\to 0$ as $i\to \infty$, so $M=L(F_{s})$, where $s$ can be worked out from the free dimension ($1+(3\alpha)/4)$).
\end{eg}
\begin{eg}  Let $A=\overset{p_{0}}{\underset{1/3}{M_{2}}}\oplus\bigoplus_{i=1}^{\infty}\overset{p_{i}}{\underset{\frac{1}{6\cdot 2^{i}}}{M_{2}}}$ and $B=\overset{q_{1}}{\underset{1/6}{M_{2}}}\oplus\overset{q_{2}}{\underset{2/3}{\bbc}}$, and $D=\overset{q_{1}}{\underset{1/3}{\bbc}}\oplus\overset{q_{2}}{\underset{2/3}{\bbc}}$, where $p_{i}\leq q_{2}$ for all $i\geq 1$, $q_{1}\leq p_{0}$, and $\tau(p_{0}q_{2})=1/3$.  Then $M=A*_{D}B$ is the amalgamated free product of two multimatrix algebras, with $G_{D}^{A,B}$ connected and $D$ finite dimensional.  The free product is:
\[
\underset{1/6}{M_{4}}\oplus\bigoplus_{i=1}^{\infty}\underset{\frac{1}{6\cdot 2^{i}}}{M_{2}}.
\]
\end{eg}
The above example contradicts the last statement of Theorem 5.1 of \cite{kenAmJM}.  That statement is erroneous and should be modified by changing ``$J$ and $X$ are finite'' to ``$J$ is finite''.
\bibliography{bib1}

\end{document}